\documentclass[pmlr]{jmlr} 

\newcommand{\tagdsmode}{submission}

\makeatletter
\newcommand{\tagdsproceedings}{proceedings}

\ifx\tagdsmode\tagdsproceedings

\else\ifx\tagdsmode\tagdssubmission
  \def\ps@jmlrtps{%
    \let\@mkboth\@gobbletwo
    \def\@oddhead{\scriptsize Under Review at the 2nd Conference on Topology, Algebra, and Geometry in Data Science\hfill}%
    \let\@evenhead\@oddhead
    \def\@oddfoot{}%
    \let\@evenfoot\@oddfoot
  }

\else
  \def\ps@jmlrtps{%
    \let\@mkboth\@gobbletwo
    \def\@oddhead{}%
    \let\@evenhead\@oddhead
    \def\@oddfoot{}%
    \let\@evenfoot\@oddfoot
  }
\fi\fi
\makeatother

\usepackage{longtable}

\usepackage{booktabs}
\usepackage[load-configurations=version-1]{siunitx} 

\usepackage{amsmath,amssymb,latexsym,color}
\usepackage{hyperref}
\usepackage{algorithm}
\usepackage{algorithmic}
\usepackage{booktabs}
\usepackage{natbib}
\usepackage{enumitem}

\theorembodyfont{\upshape}
\theoremheaderfont{\scshape}
\theorempostheader{:}
\theoremsep{\newline}

\newtheorem{assumption}{Assumption}

\jmlrvolume{334}
\jmlryear{2026}
\jmlrworkshop{Topology, Algebra, and Geometry in Data Science}

\title[Scalable Graph Coreset Selection via Greedy Sampling]{Scalable Graph Coreset Selection via Greedy Sampling}

\ifx\tagdsmode\tagdssubmission

\else

\author{\Name{Zhaiming Shen} \Email{zshen49@gatech.edu}\\
\addr School of Mathematics, Georgia Institute of Technology, Atlanta, GA 30332, USA
\AND
\Name{Alexander Cloninger} \Email{acloninger@ucsd.edu}\\
\addr Department of Mathematics and Halicio{\u g}lu Data Science Institute, University of California, San Diego, La Jolla, CA 92093, USA
}

\fi

\ifx\tagdsmode\tagdsproceedings
\editor{Editor's name}
\fi

\begin{document}

\maketitle

\begin{abstract}
Sampling representative nodes from large graphs is fundamental to graph signal processing and network analysis, yet existing methods require access to the full graph Laplacian, making them impractical at scale. We propose a simple and effective column-selective graph sampling algorithm based on a minimum inner product greedy selection rule. At each iteration, the algorithm accesses only a small random subset of Laplacian columns, requiring no eigendecomposition or global graph traversal, making it well-suited for large-scale graphs where the full Laplacian cannot be stored in memory. We analyze the algorithm under the stochastic block model and show that, when the degree distribution is balanced across nodes, the algorithm achieves sampling proportional to cluster size, and that the resulting mean estimate is controlled for band-limited graph signals in the Paley-Wiener space, with the error decaying as inter-cluster connectivity weakens. Numerical experiments on both synthetic and real-world data validate the effectiveness of the proposed method.
\end{abstract}


\section{Introduction}
\label{sec:introduction}

Graphs provide a natural representation for relational data arising in diverse domains, including social networks~\citep{newman2002random, myers2014information}, biological systems~\citep{pavlopoulos2011using, aittokallio2006graph}, and signal processing~\citep{ortega2018graph, mateos2019connecting}, where nodes encode entities and edges encode pairwise interactions. A problem that arises frequently in such settings is estimating the aggregate behavior of a function defined on a graph — specifically, computing the sum or average of function values across all nodes. For example, in elections, opinion polls are designed to estimate the views of a networked population toward candidates \citep{lippmann2017public}. In environmental monitoring, tracking regional averages of temperature, humidity, and water quality enables early intervention against forest fires and water contamination \citep{yu2005real}. In healthcare, measuring population-level indicators such as average blood pressure, weight, and cholesterol informs health policy and disease prevention strategies \citep{choi2017gram}.

In many applications, however, it is infeasible or costly to collect observations at every node, making it desirable to identify a \emph{small, representative subset of nodes}, named \emph{coreset}, that captures the essential structure of the graph. For example, deploying sensors across every node in a network may be cost-prohibitive. Similarly, collecting health measurements such as blood pressure from individuals in remote or hard-to-reach areas can be logistically challenging and expensive. This motivates the problem of coreset selection from a large graph such that a weighted sum of function values over the subset faithfully approximates the sum over the entire graph.



Coreset selection methods can be broadly classified into deterministic and random sampling approaches. Deterministic methods \citep{anis2014towards,narang2013signal,marques2015sampling,sakiyama2019eigendecomposition,cloninger2021low,vahidian2020coresets} select vertices sequentially such that a target cost function is optimized at each step, whereas random methods \citep{puy2018random,perraudin2018global} sample vertices according to a precomputed probability distribution. In recent years, coreset methods have also been studied for better training deep learning models and graph neural networks (GNNs). These include \citet{ding2024spectral,mirzasoleiman2020coresets,campbell2018bayesian,balcilar2021analyzing}.
A common limitation across these methods is their reliance on global spectral information or pairwise similarities computed over the full graph. This makes them ill-suited for large-scale settings where only partial graph access is available at each iteration, which is a practical constraint that arises frequently when graphs are too large to be processed all at once.

To address this limitation, we propose a greedy graph corset sampling method that selects nodes by
iteratively choosing the candidate with the minimum inner product
between a random subset of Laplacian columns and a running residual
vector. The algorithm is inspired by Orthogonal Matching Pursuit 
(OMP) \citep{davis1994adaptive, pati1993orthogonal}, as well as the recent compressive sensing based local clustering approaches \citep{lai2020compressive, lai2023compressed,shen2023graph,shen2025advancing}. The proposed method has two key properties that distinguish it from existing
methods.
First, it accesses only a small random subset of nodes per iteration, requiring no eigendecomposition and no global
graph traversal.
Second, the minimum inner product selection rule implicitly promotes
diversity across clusters: by choosing the node whose Laplacian column
is most orthogonal to the residual, the algorithm avoids resampling
nodes from already-represented clusters.

The main contributions of this work are as follows.
\begin{itemize}
    \item We propose a simple yet effective greedy coreset sampling method for graphs that requires only local graph information at each iteration, without ever needing access to the full graph.
    \item We theoretically establish that under the stochastic block model (SBM) with balanced degree distribution, the proposed algorithm achieves sampling proportional to cluster size in expectation, without requiring any explicit knowledge of cluster labels or sizes.
    \item We further theoretically establish that this proportional sampling guarantees accurate estimation of band-limited graph signals defined in the Paley-Wiener space.
          \item We evaluate the proposed method on both synthetic and real-world graphs under the practical setting where only a partial subgraph is accessible at each iteration.
\end{itemize}

\paragraph{Organization.}
The rest of the paper is structured as follows. Section~\ref{sec:background} reviews the necessary background on graph signal processing and Paley-Wiener spaces on graphs. Section~\ref{sec:algorithm} presents the proposed sampling algorithm and discusses its computational and memory efficiency. Section~\ref{sec:theory} establishes the theoretical guarantees of the proposed method under the stochastic block model. Section~\ref{sec:experiments} reports experimental results on both synthetic and real-world datasets. Section~\ref{sec:conclusion} concludes with a summary and directions for future work.

\section{Preliminaries}
\label{sec:background}


\subsection{Graph Laplacian and its Spectral Decomposition}
 
Let $G = (V, E)$ be an undirected, unweighted graph with node set
$V = \{1, \ldots, n\}$ and edge set $E \subseteq V \times V$.
The adjacency matrix $A \in \{0,1\}^{n \times n}$ has entries
$A_{ij} = 1$ if $(i,j) \in E$ and $A_{ij} = 0$ otherwise.
The degree of node $i$ is $d_i = \sum_{j=1}^n A_{ij}$, collected into
the diagonal degree matrix $D = \mathrm{diag}(d_1, \ldots, d_n)$.
The \emph{unormalized graph Laplacian} is defined as $L:=D-A$.
The columns of $L$ play a central role in the proposed algorithm:
the $i$-th column $L_{:,i}$ encodes the local connectivity structure
of node $i$, with the diagonal entry reflecting its degree and the
off-diagonal entries indicating its neighbors.

Since $L$ is real symmetric, it admits an eigendecomposition:
\begin{equation}
    L = U \Lambda U^\top = \sum_{k=0}^{n-1} \lambda_k u_k u_k^\top,
    \label{eq:spectral}
\end{equation}
where $U = [u_0 \mid u_1 \mid \cdots \mid u_{n-1}]$ is the orthonormal
matrix of eigenvectors and $\Lambda = \mathrm{diag}(\lambda_0, \ldots,
\lambda_{n-1})$ contains the eigenvalues sorted in ascending order
$0 = \lambda_0 \leq \lambda_1 \leq \cdots \leq \lambda_{n-1}$.
The smallest eigenvector is always $u_0 = \frac{1}{\sqrt{n}}\mathbf{1}$
with eigenvalue $\lambda_0 = 0$.
The number of zero eigenvalues equals the number of connected components
of $G$.
 
 
\subsection{Graph Signal Processing and Paley-Wiener Space}
 
A \emph{graph signal} is a function $f : V \to \mathbb{R}$, represented
as a vector $f \in \mathbb{R}^n$ with $f(i)$ denoting the signal value
at node $i$. The \emph{graph Fourier transform} (GFT) of $f$ is:
\begin{equation}
    \hat{f} = U^\top f, \qquad \hat{f}_k = u_k^\top f,
    \label{eq:gft}
\end{equation}
with inverse $f = U\hat{f} = \sum_{k=0}^{n-1}\hat{f}_k u_k$.
The eigenvalue $\lambda_k$ plays the role of graph frequency: signals
concentrated on low-frequency eigenvectors (small $\lambda_k$) vary
slowly across the graph, while high-frequency components (large
$\lambda_k$) vary rapidly between adjacent nodes.
 
 
The \emph{Paley-Wiener space} on the graph is the space of
bandlimited signals with graph frequency below $\omega$:
\begin{equation}
    PW_\omega(L) = \bigl\{\, f \in \mathbb{R}^n :
    \hat{f}_k = 0 \text{ whenever } \lambda_k > \omega \,\bigr\}.
    \label{eq:pw}
\end{equation}
For $\omega$ chosen in the spectral gap $(\lambda_{K-1}, \lambda_K)$,
the space $PW_\omega(L)$ equals the span of the first $K$ eigenvectors:
$PW_\omega(L) = \mathrm{span}\{u_0, \ldots, u_{K-1}\}$.

\section{Scalable Graph Coreset Selection} \label{sec:algorithm}

\subsection{Proposed Method}

The proposed algorithm adopts a greedy strategy to iteratively construct a coreset by selecting one node each iteration. We describe the procedure below and summarize it in \algorithmref{alg_omp}. The following steps repeat until $T$ iterations.
\begin{enumerate}
    \item It maintains a residual vector $\mathbf{b} \in \mathbb{R}^n$, initialized to zero, which tracks the cumulative influence of already-selected nodes.
  
    \item At each iteration, a small candidate set $\mathcal{C}_t$ nodes is drawn uniformly at random.
    
    \item Computes the absolute inner product between each candidate column of $L$ and the current residual $\mathbf{b}$, measuring how aligned each candidate node is with the accumulated selection. 

    \item The node with the minimum inner product, indicating the least redundancy with previously selected nodes, is chosen, with ties broken uniformly at random. 
 
    \item The selected node is added to the coreset $\mathcal{X}$, the residual is updated by subtracting the corresponding column of $L$, and the selected column is set to $\infty$ 
to prevent reselection.
\end{enumerate}







\begin{algorithm}[t]
\caption{Greedy Graph Coreset Selection (GGCS) \label{alg_omp}} 
\begin{algorithmic}[1]
\REQUIRE unnormalized graph Laplacian $L \in \mathbb{R}^{n \times n}$, number of iterations $T$
\ENSURE Sampled node index set $\mathcal{X}$

\STATE $\mathbf{b} \leftarrow \mathbf{0} \in \mathbb{R}^{n}$, $\mathcal{X} \leftarrow \emptyset$ 
\FOR{$t = 1, \ldots, T$}
    \STATE Sample a small candidate set $\mathcal{C}_t \subset \{1, \ldots, n\}$ uniformly at random
    \STATE Compute inner products:
        $\mathbf{a} = \left| L_{:,\mathcal{C}_t}^\top \mathbf{b} \right| \in \mathbb{R}^{|\mathcal{C}_t|}$
    \STATE Find minimum inner product value:
        $m^* = \min_{j \in \mathcal{C}_t} \mathbf{a}_j$    
    \STATE Collect all minimizers:
        $\mathcal{I} = \left\{ j \in \mathcal{C}_t : \mathbf{a}_j = m^* \right\}$
and break ties uniformly at random: $i^* \leftarrow \text{UniformSample}(\mathcal{I})$
    \STATE Update selected set: $\mathcal{X} \leftarrow \mathcal{X} \cup \{i^*\}$ and update residual:
        $\mathbf{b} \leftarrow \mathbf{b} - L_{:,i^*}$
    \STATE Eliminate selected column: $L_{:,i^*} \leftarrow \boldsymbol{\infty}$ 
\ENDFOR
\end{algorithmic}
\end{algorithm}



\subsection{Scalability via Partial Graph Access}
\label{subsec:scalability}
 
A central advantage of the proposed algorithm over spectral and
deterministic methods is that it \emph{never accesses the full graph
Laplacian}. At each iteration $t$, it draws a random candidate set
$\mathcal{C}_t$ of size $s \ll n$ and reads only the $s$ columns
$L_{:,\mathcal{C}_t}$. No eigendecomposition, no pairwise kernel
matrix, and no global graph traversal is required.  Note that although the entire column $L_{:i}$ is sampled, it only encodes local information about which nodes share an edge with node $i$.  This can be readily accessed in relational databases without incurring substantial overhead or memory usage.

\paragraph{Time cost.}
At each iteration $t$, the computational work consists of three
steps. Column retrieval reads $s$ columns of $L$, each with at most
$d_{\max}$ nonzero entries, at cost $O(s \cdot d_{\max})$. Inner
product computation evaluates $\mathbf{a} = |L_{:,\mathcal{C}_t}^\top
\mathbf{b}|$ at the same cost $O(s \cdot d_{\max})$, since
$\mathbf{b} \in \mathbb{R}^n$ but each column has $O(d_{\max})$
nonzeros. The residual update $\mathbf{b} \leftarrow \mathbf{b} -
L_{:,i^*}$ reads one column with the worse case cost  $O(d_{\max})$. The total time cost over $T$
iterations is therefore $O\!\left(T \cdot s \cdot d_{\max}\right)$.

\paragraph{Memory cost.}
The memory requirement is determined by what must be held
\emph{simultaneously}, not what is accessed in total. At any point
in time the algorithm stores:
(i). The residual vector $\mathbf{b} \in \mathbb{R}^n$ with $O(n)$. (ii). The current column batch $L_{:,\mathcal{C}_t}$ with $O(s \cdot d_{\max})$ for sparse $L$. (iii). The inner product vector $\mathbf{a} \in \mathbb{R}^s$ with $O(s)$.

The $s$ columns in $\mathcal{C}_t$ are read, their inner products
computed, and then discarded — the full Laplacian $L$ is
\emph{never stored in memory}. The peak memory cost is therefore
$O\!\left(n + s \cdot d_{\max}\right)= O(n)$,
where the $O(n)$ term from $\mathbf{b}$ dominates when
$s \cdot d_{\max} \leq n$, which holds whenever the candidate
batch is much smaller than the full graph.
 
\paragraph{The key advantage.}
The memory cost $O(n)$ is the critical scalability gain: the
algorithm requires memory linear in the number of nodes, not
quadratic. This makes it applicable to graphs where $n$ is large
enough that even storing $L$ is infeasible. 

\section{Theoretical Analysis}
\label{sec:theory}

Let us present the theoretical analysis of  \algorithmref{alg_omp} under the stochastic block model Assumptions~\ref{ass:sbm}--\ref{ass:balanced}. It is worthwhile to note that
Assumptions~\ref{ass:sparse}--\ref{ass:balanced} are the key structural conditions under which the proposed algorithm achieves cluster-size-proportional sampling. Due to space constraints, we defer these assumptions to \appendixref{sec:assump}.

\subsection{Balanced Cluster Proportion}

The key quantity governing \algorithmref{alg_omp} is the inner product between
Laplacian columns, which equals the $(i,j)$ entry of $L^2$. Our first \lemmaref{lem:l2_exact} characterizes the inner product between columns $i$ and $j$ of the Laplacian depending on whether they belong to the same cluster or to different clusters.
 
\begin{lemma}[Exact $L^2$ Entry and SBM Expectations]
\label{lem:l2_exact}
For any $i \neq j$:
\begin{equation}
    (L^2)_{ij} = |\mathcal{N}(i) \cap \mathcal{N}(j)|
    - (d_i + d_j)\,\mathbf{1}[(i,j) \in E],
    \label{eq:l2_exact2}
\end{equation}
where $\mathcal{N}(i) = \{k \neq i : (i,k)\in E\}$ and $d_i =
|\mathcal{N}(i)|$. Under Assumptions~\ref{ass:sbm}--\ref{ass:sparse}:
\begin{equation}
    \mathbb{E}[(L^2)_{ij}] = \begin{cases}
        -n_a p_a^2\!\left(1 + O\!\left(\dfrac{q_{\max}}{p_a}\right)\right)
        & i,j \in V_a, \\[10pt]
        \displaystyle\sum_{c \neq a,b} n_c q_{ac} q_{bc}
        + O(n q_{\max}^2) = O(n q_{\max}^2)
        & i \in V_a,\, j \in V_b,\, a \neq b.
    \end{cases}
    \label{eq:l2_simplified}
\end{equation}
\end{lemma}
 

\algorithmref{alg_omp} maintains a residual $\mathbf{b}_t =
-\sum_{s=1}^{t} L_{:i_s}$ and selects the node with minimum
$|L_{:j}^\top \mathbf{b}_t|$ from a random candidate set.
The following \theoremref{thm:residual} derives an exact expression for the expected inner product at each iteration and establishes that the algorithm samples nodes in proportion to cluster size based on \lemmaref{lem:l2_exact}.

\begin{theorem} [Sampling Proportional to Cluster Size]\label{thm:residual}
Under Assumptions~\ref{ass:sbm}--\ref{ass:balanced}, for any candidate
node $j \in V_a$ and any deterministic selected set
$\mathcal{X}_t = \{i_1,\ldots,i_{t}\}$ at $t$-th iteration, it satisfies at $(t+1)$-th iteration:
\begin{equation}
    \mathbb{E}\!\left[L_{:,j}^\top \mathbf{b}_t
    \,\middle|\, \mathcal{X}_t\right]
    = m_a^{(t)} \cdot n_a p_a^2 + \eta_a^{(t)},
    \label{eq:residual_exact}
\end{equation}
where $m_a^{(t)} = |\mathcal{X}_t \cap V_a|$ and the remainder
satisfies:
\begin{equation}
    |\eta_a^{(t)}|
    \leq m_a^{(t)} \cdot O(n p_a q_{\max})
    + \sum_{c \neq a} m_c^{(t)} \cdot O(n q_{\max}^2).
    \label{eq:eta_bound}
\end{equation}
Consequently,  $|\eta_a^{(t)}|=o(m_a^{(t)} n_a p_a^2)$. Furthermore, suppose that $\frac{m_a^{(t)}}{n_a}$ is the smallest among all $a=1,\cdots,K$ up to tolerance $|\eta_a^{(t)}|$, then in expectation \algorithmref{alg_omp} selects a column in $V_a\cap\mathcal{C}_{t+1}$ in the $(t+1)$-th iteration. 
\end{theorem}

\theoremref{thm:residual} shows that $|L_{:,j}^\top\mathbf{b}_t|$ is
determined to leading order by $m_a^{(t)}\cdot n_ap_a^2$, with a
multiplicative correction $O(q_{\max}/p_a)$ that vanishes under
Assumption~\ref{ass:sparse}. Intuitively, \theoremref{thm:residual} guarantees that \algorithmref{alg_omp} always favors selecting a node from the cluster that is currently most underrepresented, thereby maintaining sampling proportional to cluster size.

\begin{figure}[t]
    \centering
    \begin{tabular}{ccc}
      \vspace{-7mm}\includegraphics[width=0.32\linewidth]{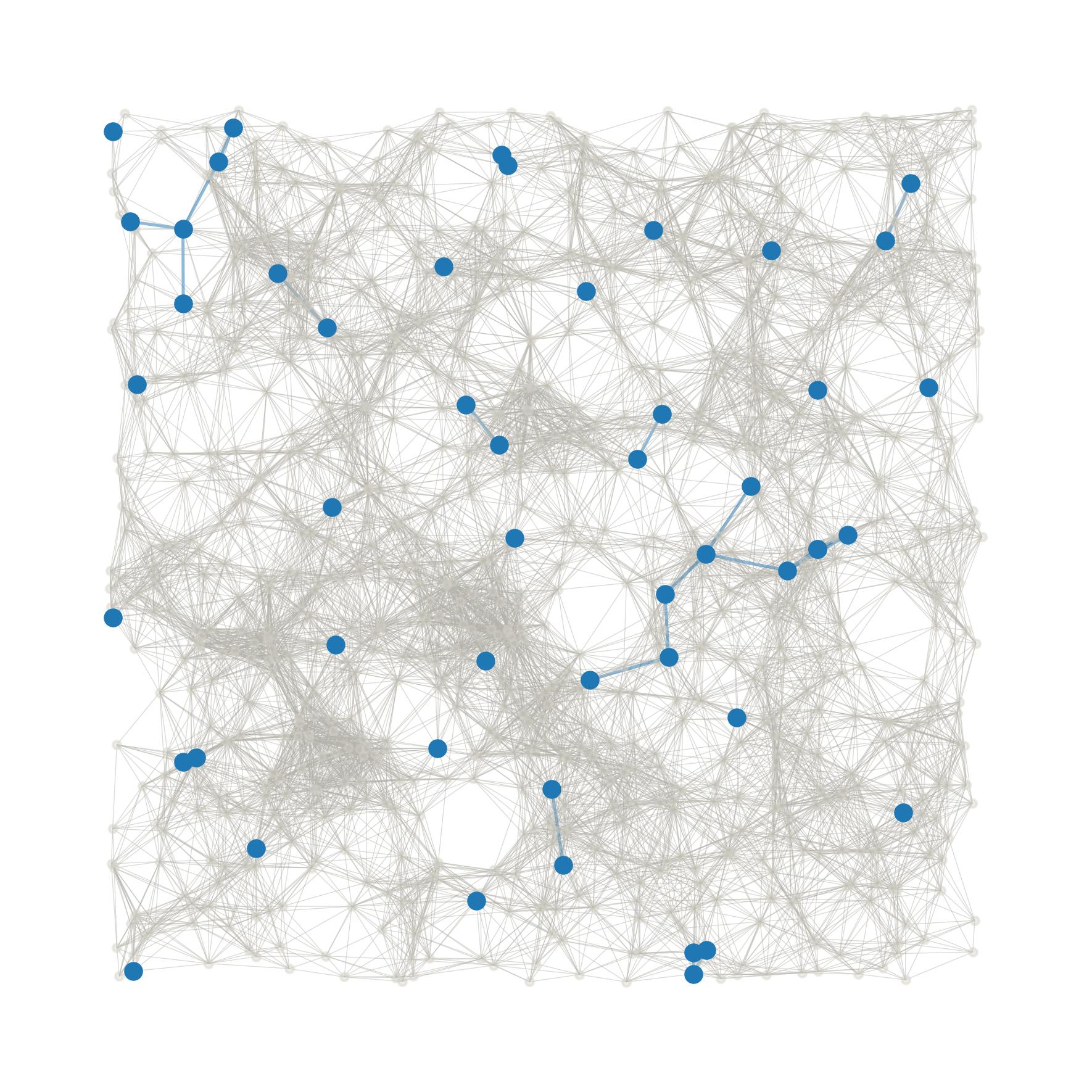}   & \includegraphics[width=0.32\linewidth]{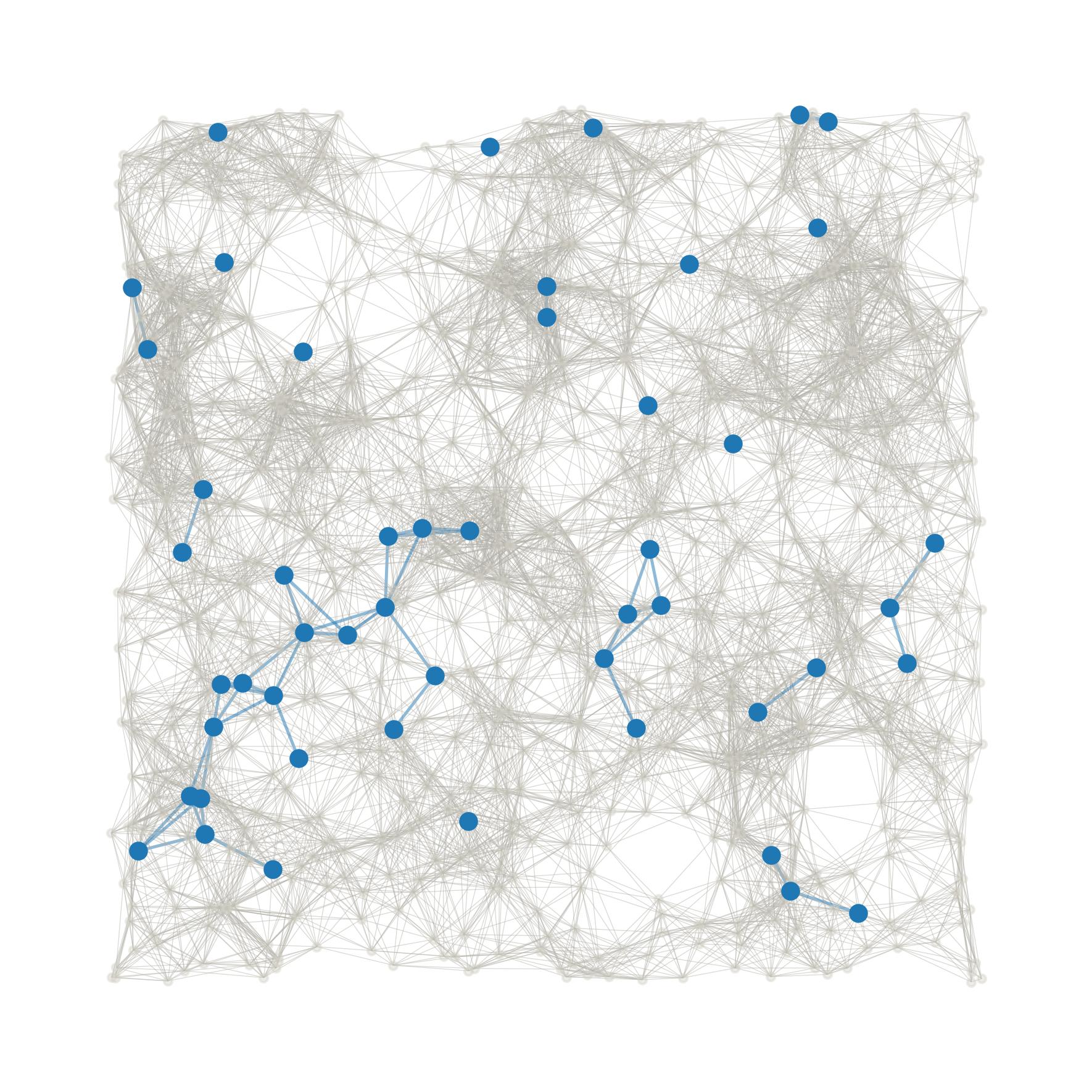} & \includegraphics[width=0.32\linewidth]{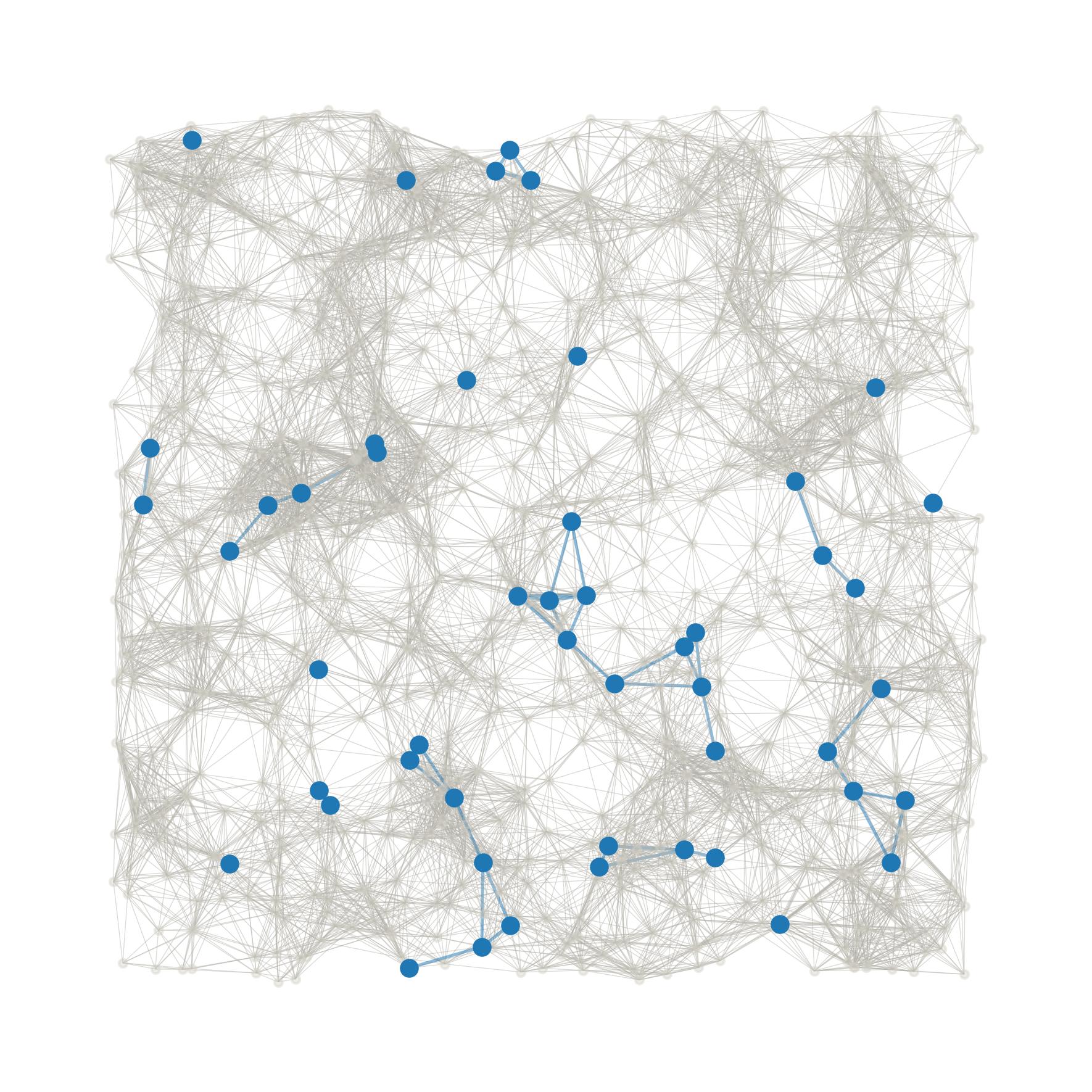} \\ \includegraphics[width=0.32\linewidth]{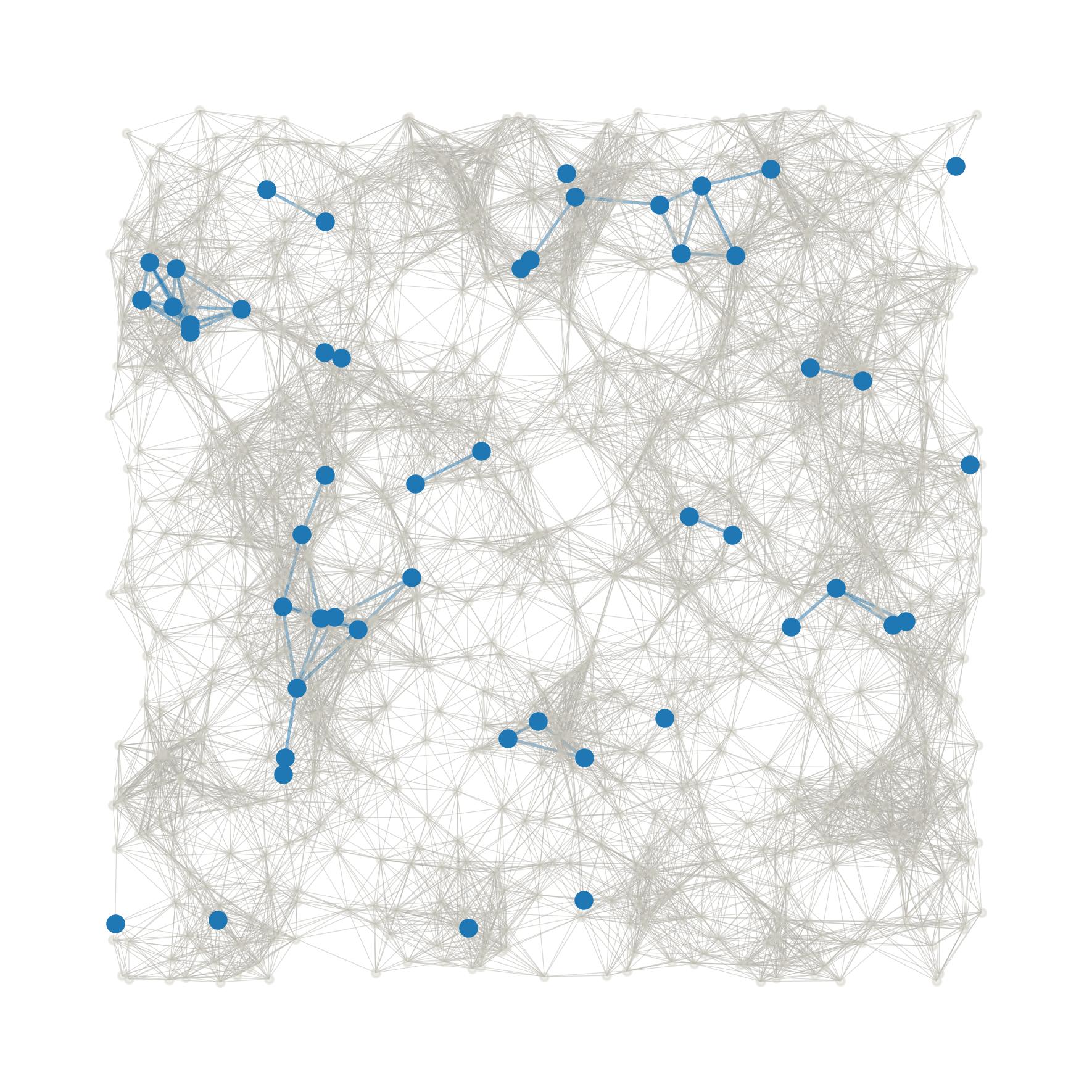}   & \includegraphics[width=0.32\linewidth]{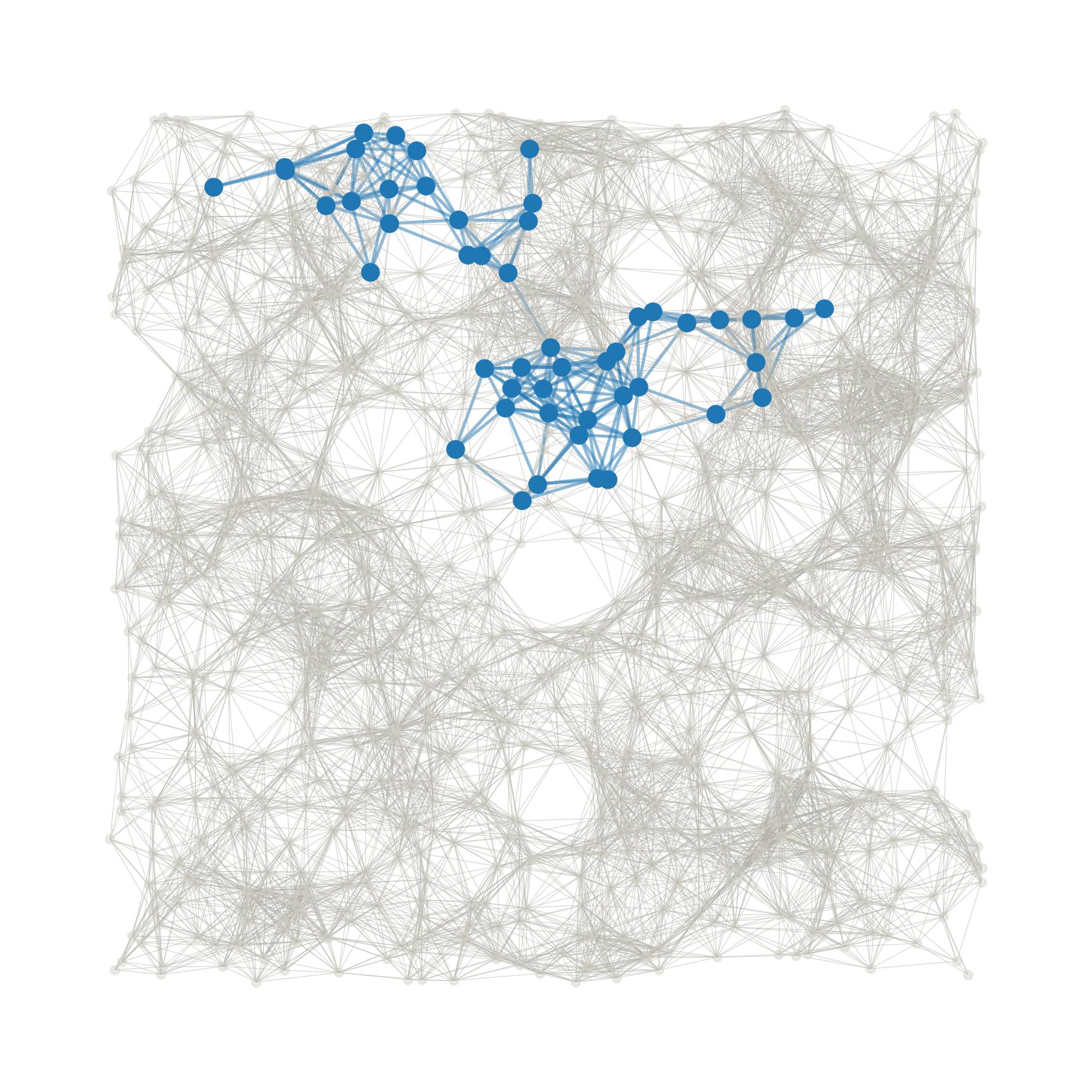} & \includegraphics[width=0.32\linewidth]{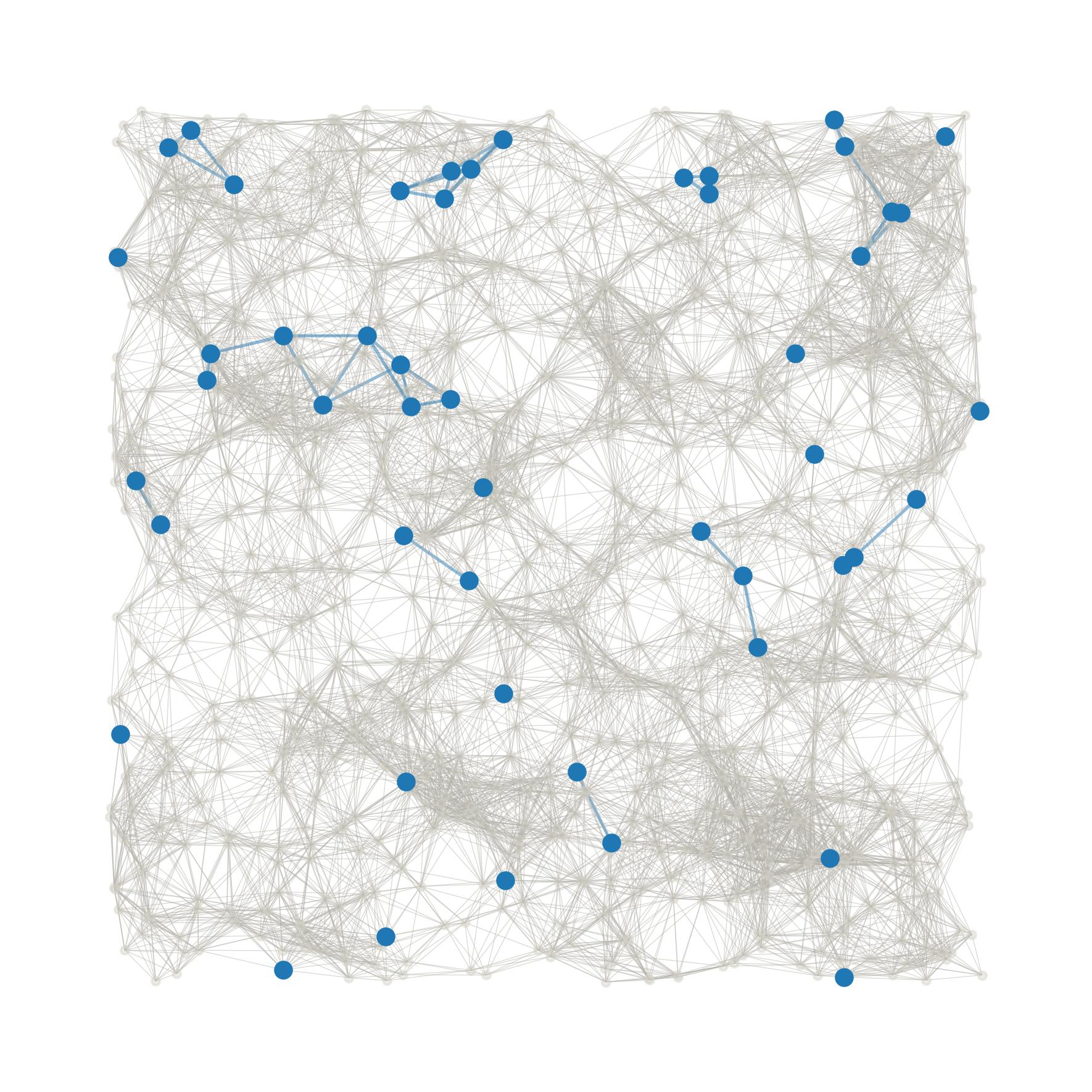}
    \end{tabular}
    \vspace{-8mm}
    \caption{Examples of various sampling schemes for $50$ nodes on random geometric graph generated under case $(ii)$. (Top row from left to right: GGCS, Random Node, Page Rank. Bottom row from left to right: Degree-based, MHRW, uniform).}
    \label{fig:RGG}
\end{figure}

\subsection{Mean Estimation of Graph Signal on Coreset}

Let us now connect the balanced cluster proportion to the signal estimation
quality, showing that proportional coverage is sufficient for good
mean estimation of band-limited functions. 

Write $L = L^{\rm block}+E$, where $L^{\rm block}$
is the block diagonal matrix encoding all intra-cluster information and $E$ captures the inter-cluster connections. 
Our goal is to bound the absolute error between the sample mean
$\hat{\mu}_S := \frac{1}{T}\sum_{i \in S} f(i)$
and the population mean $\mu := \frac{1}{n}\sum_{i \in V} f(i)$.

\begin{theorem}[Mean Estimation Error for Band-limited Signals]
\label{thm:error_bound}
Let $f \in PW_\omega(L)$ for $\omega \in (\lambda_{K-1}, \lambda_K)$,
so that $f = \sum_{k=0}^{K-1}\hat{f}_k u_k$.
Let $S$ be a coreset of size $T$ obtained from  \algorithmref{alg_omp} with per-cluster counts satisfying
$m_a/T = n_a/n + \epsilon_a$ with $|\epsilon_a| \leq \epsilon$ for all
$a \in \{1,\ldots,K\}$.
Let $\Delta = \lambda_K - \lambda_{K-1}$ denote the spectral gap of $L$.
Then:
\begin{equation}
    |\hat{\mu}_S - \mu|
    \leq \sqrt{K}\,\|f\|_2\,\left(\epsilon\sqrt{2K}+\frac{4\|E\|_2}{\Delta\cdot \sqrt{T}}\right).
    \label{eq:error_bound}
\end{equation}
The first term reflects cluster imbalance
in the coreset and the second reflects eigenvector perturbation due to
cross-cluster edges.
\end{theorem}
 
\theoremref{thm:error_bound} shows that the coreset mean estimate is accurate provided that the sampled cluster proportions are balanced and the inter-cluster connections are sparse.

\begin{figure}[t]
    \centering
    \begin{tabular}{ccc}
      \vspace{-8mm}\includegraphics[width=0.32\linewidth]{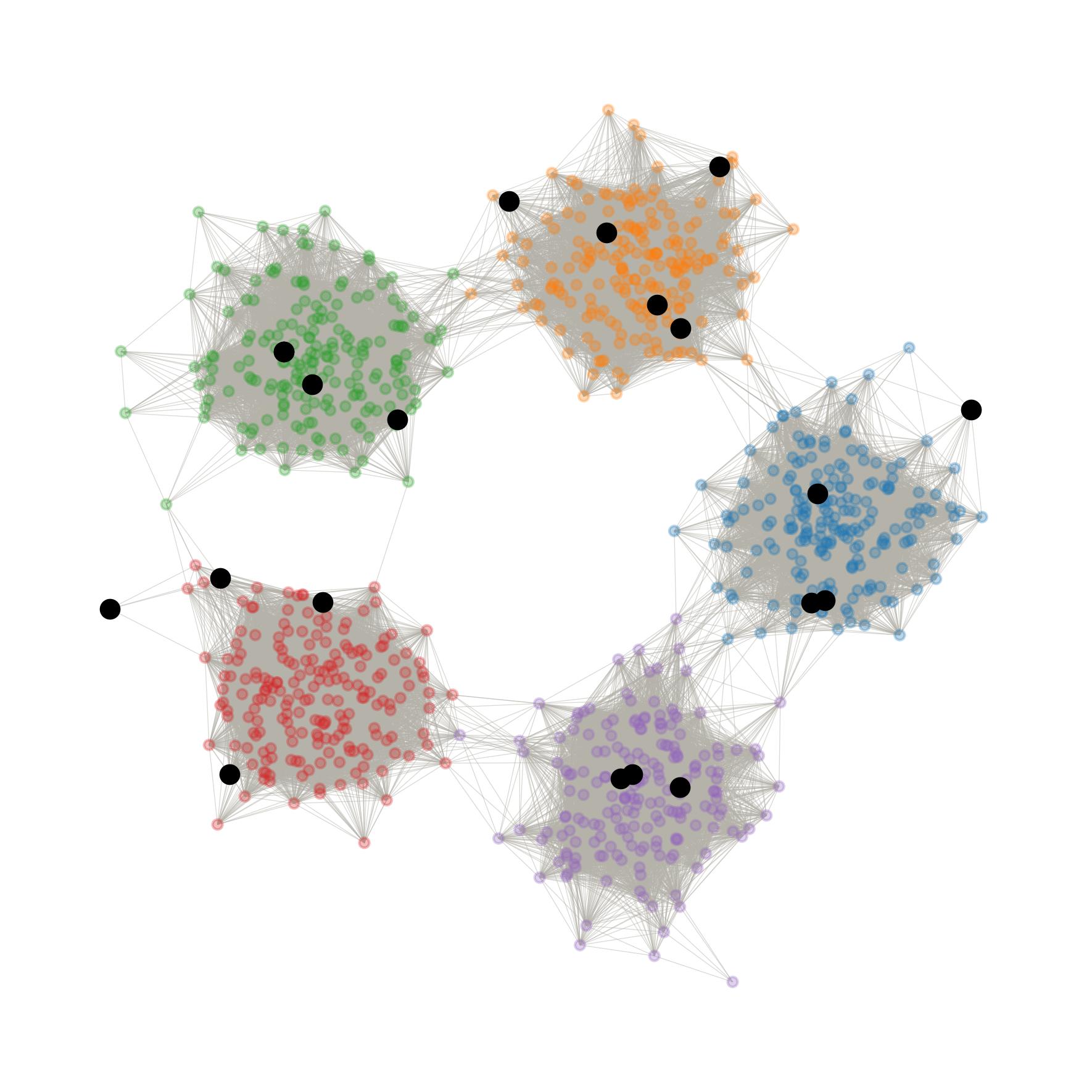}   & \includegraphics[width=0.32\linewidth]{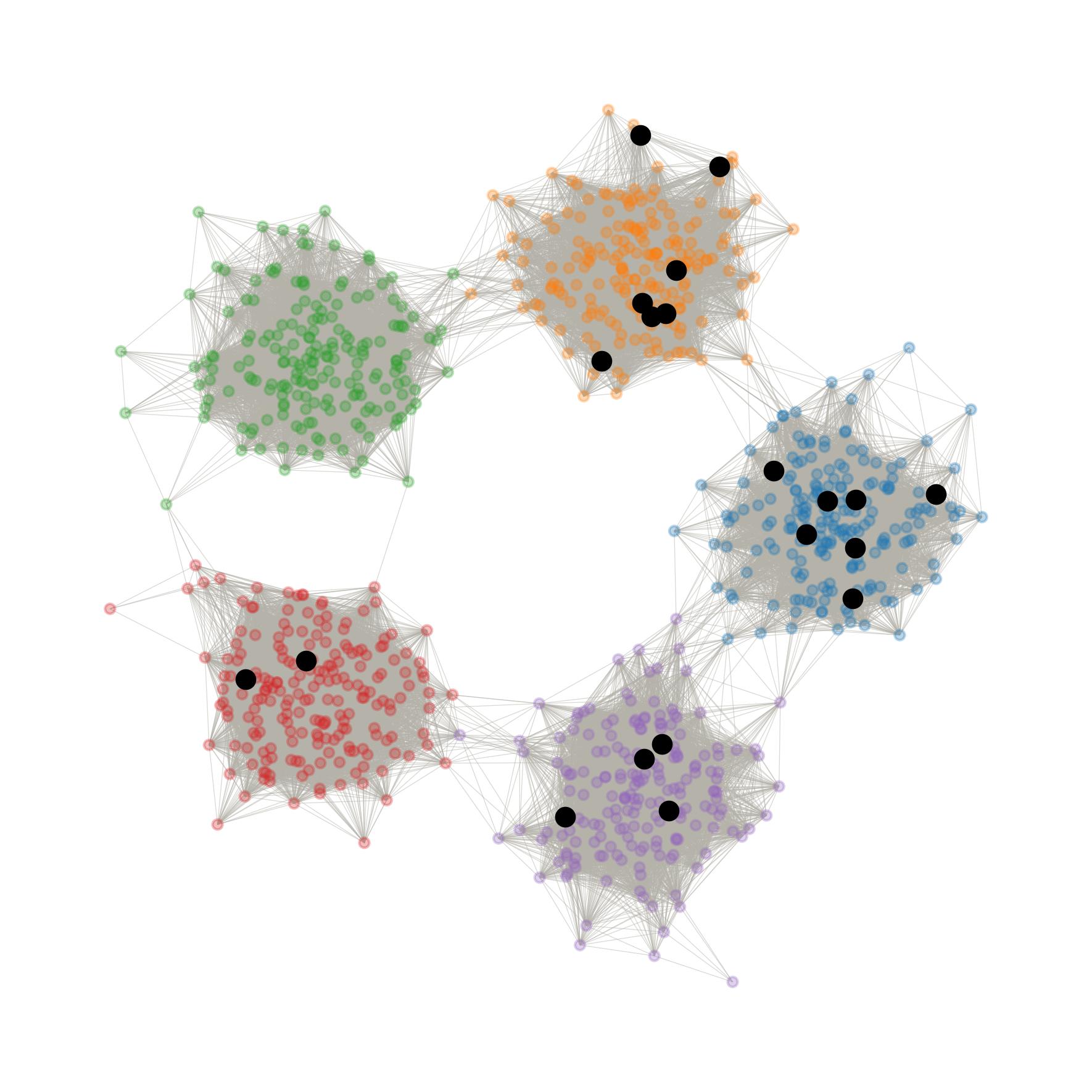} & \includegraphics[width=0.32\linewidth]{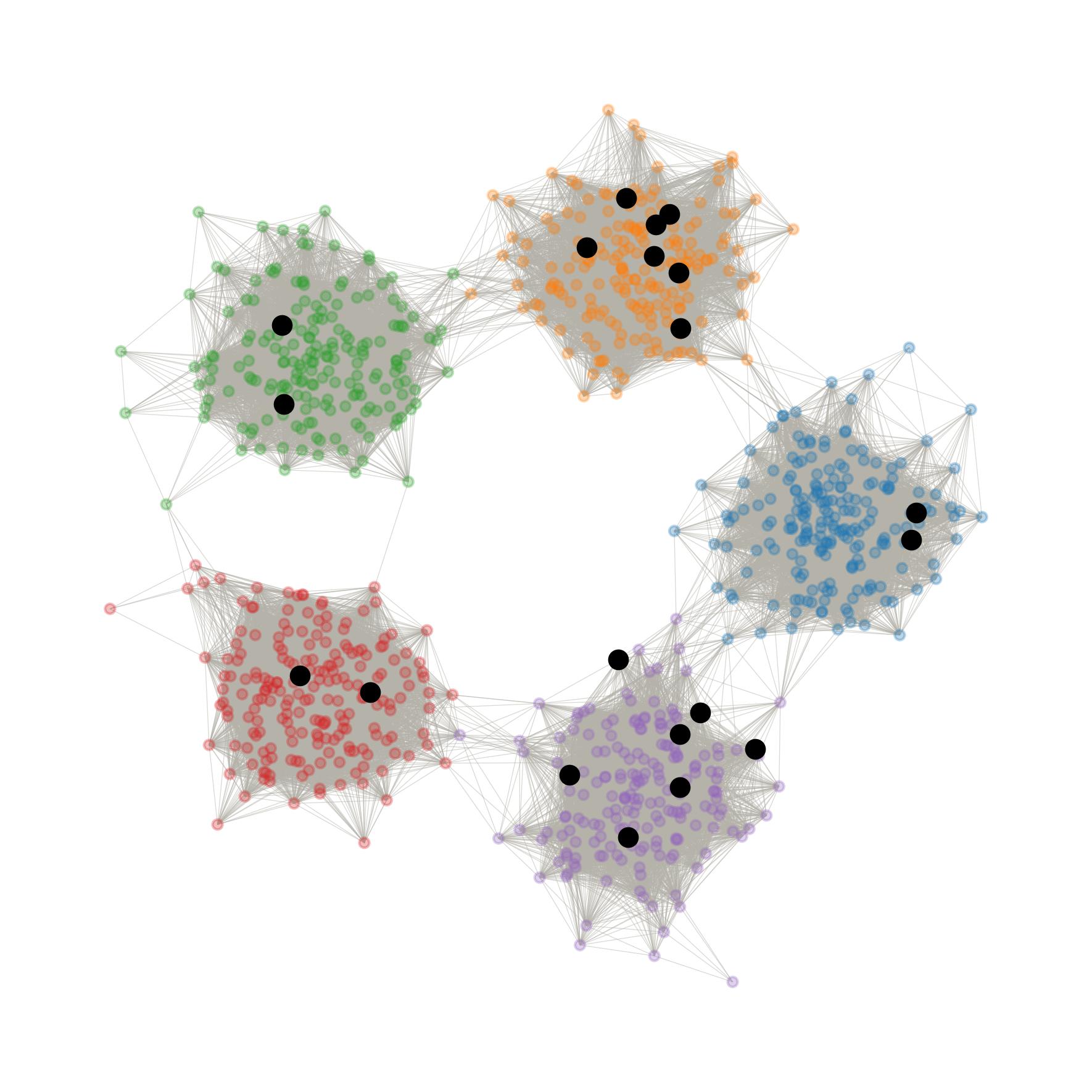} \\ 
      \includegraphics[width=0.32\linewidth]{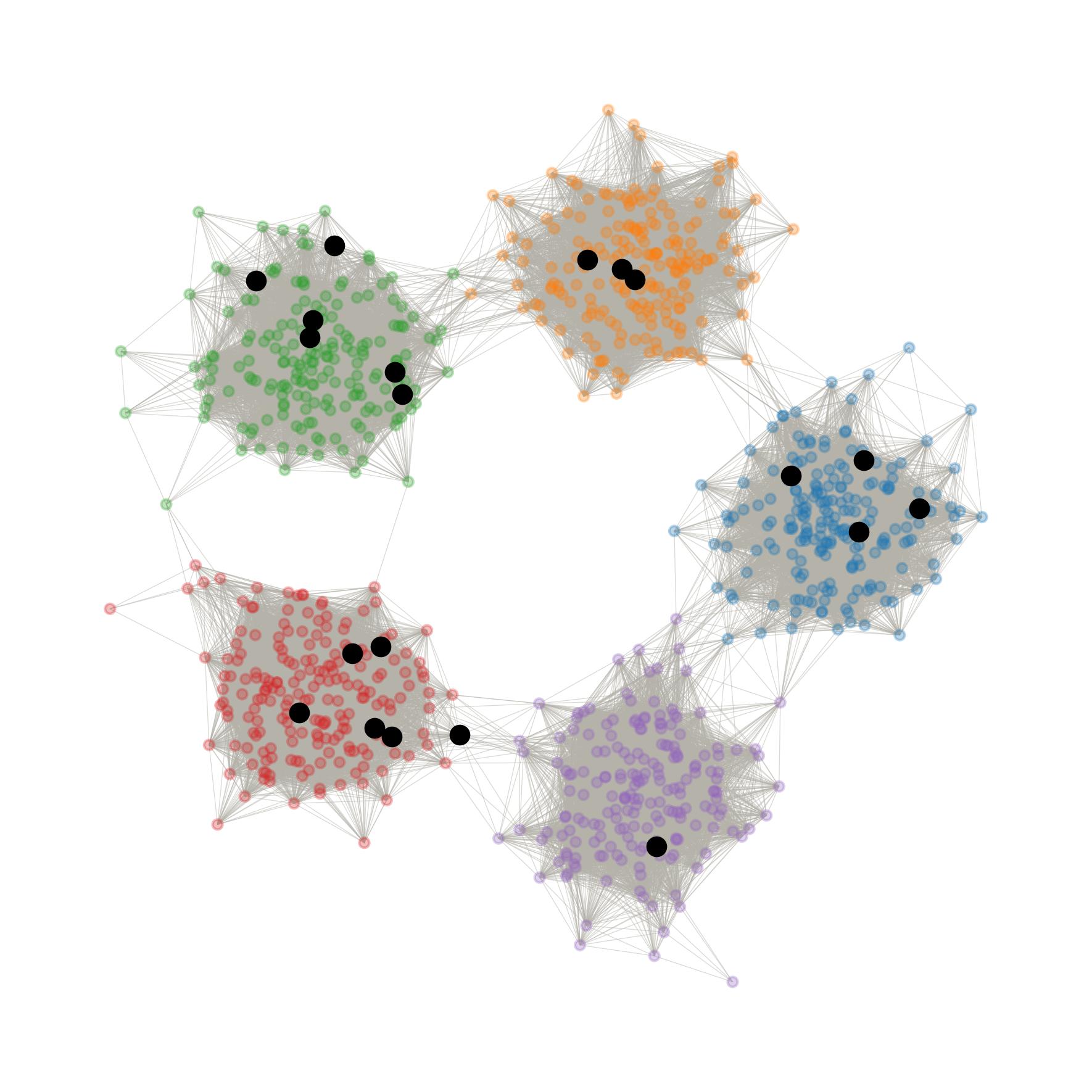}   & \includegraphics[width=0.32\linewidth]{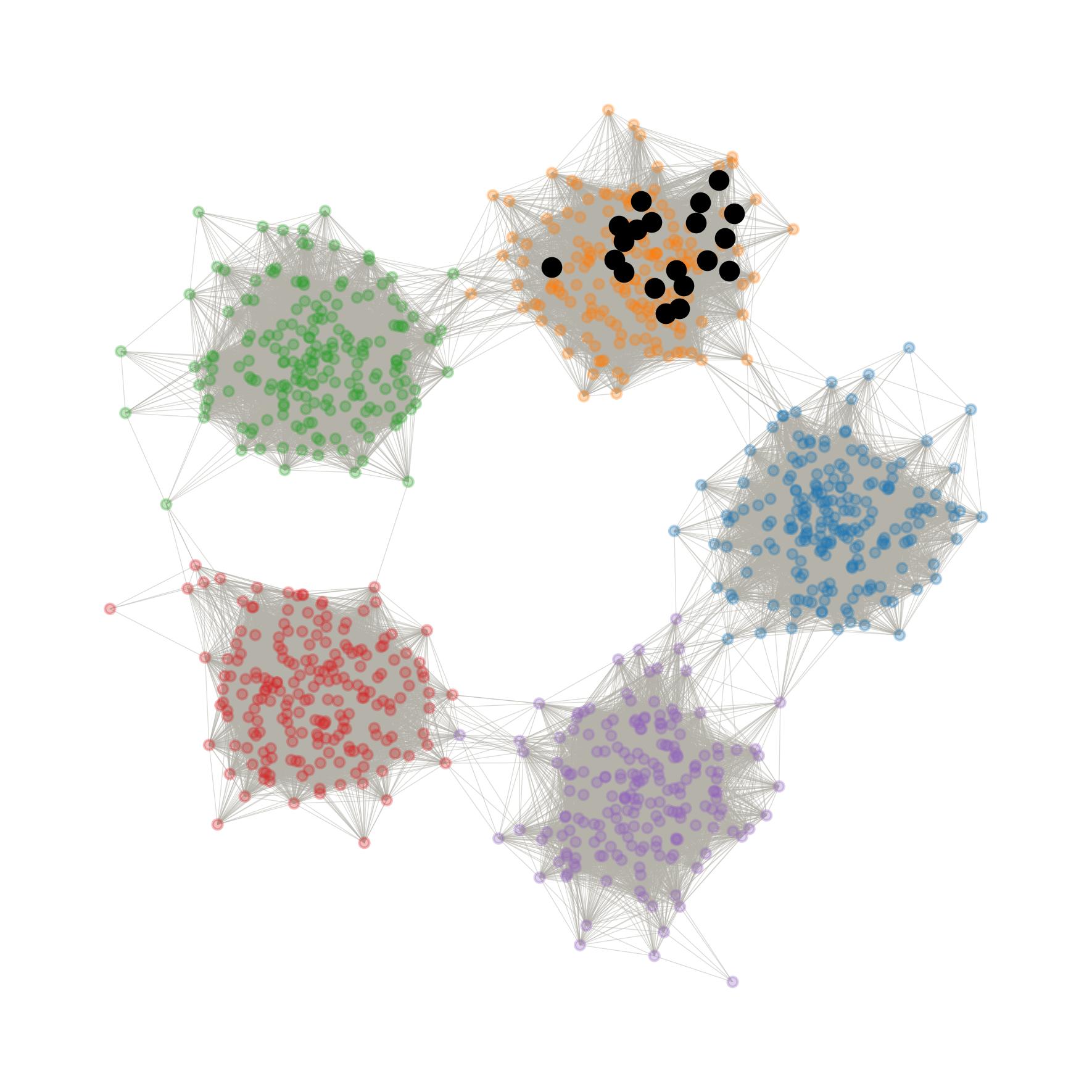} & \includegraphics[width=0.32\linewidth]{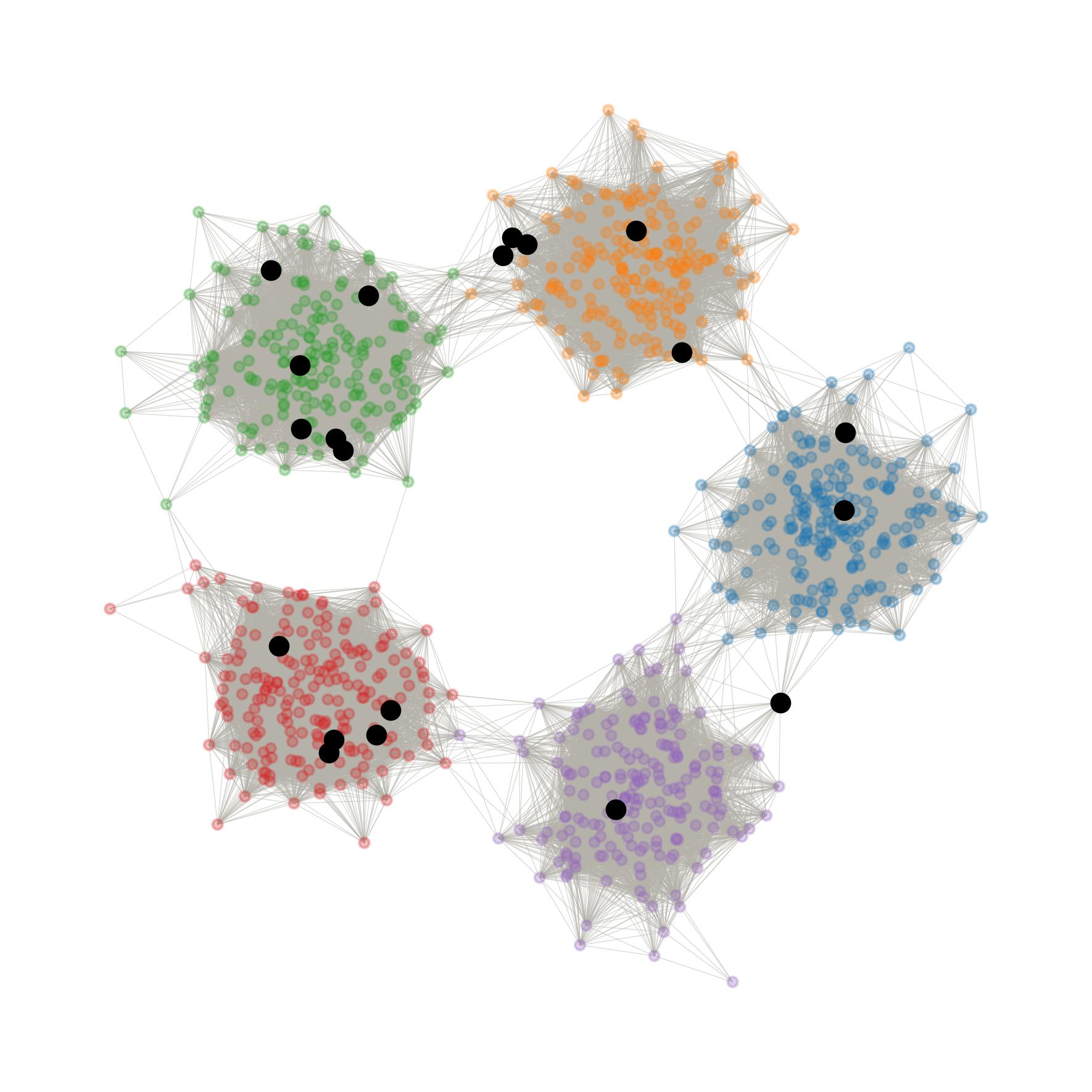}
    \end{tabular}
    \vspace{-9mm}
    \caption{Examples of various sampling schemes for $20$ nodes on random geometric graph generated under case $(iii)$. (Top row from left to right: GGCS, Random Node, Page Rank. Bottom row from left to right: Degree-based, MHRW, uniform).}
    \label{fig:RGG_cluster}
\end{figure}

\section{Experiments}
\label{sec:experiments}
 

\paragraph{Datasets.} For the synthetic dataset, we consider three types of synthetic random graphs: $(i)$ $\mathrm{SBM}(n, K, \{n_a\}_{a=1}^K, p, q)$
with $n$ nodes partitioned into $K$ clusters of sizes $n_1, \ldots, n_K$. $(ii)$ Random geometric graph in 2D, which places $n$ nodes uniformly at random in $[0,1]^2$, connects any two nodes within Euclidean distance $0.1$. $(iii)$ Random geometric graph in 2D, but with a more clean spatial clustering structure.

For the real-world datasets, we consider: $(iv)$ Cora \citep{sen2008collective}, which is a citation digraph $2708$ nodes, and contains $7$ different classes. We treat it as an undirected graph and take the largest connected component with $2485$ nodes. $(v)$  Amazon Photo \citep{shchur2018pitfalls} is a co-purchase graph derived from Amazon, consisting of $7650$ nodes spanning $8$ product categories, where edges connect items frequently bought together.

\paragraph{Benchmark Methods.}  
Deferred to \appendixref{sec:exp_details}.


 

\paragraph{Testing Functions and Evaluation Metrics.}  We evaluate the proposed coreset selection algorithm using \emph{Band-limited Function}.
A band-limited function defined on a graph is a random linear combination of the first
$K$ eigenvectors of $L$:
\begin{equation}
    f = \sum_{k=0}^{K-1} \alpha_k u_k,
    \quad \alpha_k \sim \mathcal{U}(0, 1) \text{ i.i.d.},
    \label{eq:pw_random}
\end{equation}
where $\mathcal{U}(0,1)$ denotes the uniform distribution on $[0,1]$. We multiply the function value by $\sqrt{n}$ for all cases to avoid small values of $f$. 

\paragraph{Testing Functions and Evaluation Metrics.} 
Given a sampled coreset $S$ of size $T$, we evaluate the following metrics: 
\begin{itemize}
    \item \emph{Mean absolute error}:
The primary metric is the absolute error between the sample mean
$\hat{\mu}_S = \frac{1}{T}\sum_{i \in S} f(i)$
and the population mean $\mu = \frac{1}{n}\sum_{i \in V} f(i)$:
\begin{equation}
    \epsilon(S) = \bigl|\hat{\mu}_S - \mu\bigr|.
    \label{eq:error}
\end{equation}
\item \emph{Cluster balance error}:
We report the per-cluster sampling count $m_a = |S \cap V_a|$
and compare against the proportional target $T n_a / n$.
The cluster balance error is:
\begin{equation}
    \mathcal{B}(S) = \frac{1}{K}\sum_{a=1}^{K}
    \left|\frac{m_a}{T} - \frac{n_a}{n}\right|.
    \label{eq:balance}
\end{equation}
\end{itemize}


\begin{table}[t]
\centering
\caption{Average $\epsilon(S)$ and $\mathcal{B}(S)$ on sampling $100$ nodes from the SBM graph with three clusters of
sizes $n_1 = 2000$, $n_2=1000$, $n_3 = 500$, and intra-cluster edge
probabilities $p_1, p_2, p_3 = 0.01, 0.02, 0.04$, and $q_{ij}=0.001$ for $i,j=1,2,3$.
The signal $f$ is defined in \eqref{eq:pw_random}.  }
\label{tab:mean_estimation_sbm1}
\vspace{-2mm}
\begin{tabular}{ccccccc}
\toprule
     & 
Random Node &
Page Rank &
Degree Based & 
MHRW &  Uniform & GGCS (ours) \\
\midrule
$\epsilon(S)$ & 0.061   & 0.057  & 0.087 & 0.17  & 0.081  & \textbf{0.024}  \\
$\mathcal{B}(S)$  & 0.037 & 0.032 & 0.033 & 0.10  & 0.033 & \textbf{0.030} 
\\
\bottomrule
\end{tabular}
\end{table}

\begin{table}[t]
\centering
\caption{Average results on sampling $50$ nodes from the random geometric graph with $n=1000$.
The signal $f$ is defined in \eqref{eq:pw_random}. The mean absolute error $\epsilon(S)$ is calculated under case $(ii)$, and the class imbalance error $\mathcal{B}$(S) is calculated under case $(iii)$. For the first row, the average true $|\mu|$ equals to $0.477$. }
\label{tab:mean_estimation_rgg}
\vspace{-2mm}
\begin{tabular}{ccccccc}
\toprule
     & 
Random Node &
Page Rank &
Degree Based & 
MHRW &  Uniform & GGCS (ours) \\
\midrule
$\epsilon(S)$ &  0.084 & 0.086 & 0.084 & 0.50 & 0.071 & \textbf{0.067} \\
$\mathcal{B}(S)$ & 0.042  & 0.043 & 0.045 & 0.30 & 0.047 & \textbf{0.038}
\\
\bottomrule
\end{tabular}
\end{table}

\subsection{Results}
For synthetic data, under $(i)$, we summarize the results for different sampling schemes in \tableref{tab:mean_estimation_sbm1} and \tableref{tab:mean_estimation_sbm2}. Under $(ii)$ and $(iii)$, we summarize the results in \tableref{tab:mean_estimation_rgg}, and we also illustrate different sampling schemes in \figureref{fig:RGG} and \figureref{fig:RGG_cluster}. \figureref{fig:RGG} shows that nodes selected by competing methods tend to cluster together or form elongated paths, whereas those selected by our method are more evenly distributed across the graph. \figureref{fig:RGG_cluster} further confirms that our method achieves more balanced selection across clusters compared to the baselines. 

For real-world data, we summarize the class balance errors for different sampling sizes in \tableref{tab:mean_estimation_real1} and \tableref{tab:mean_estimation_real2}. We compute the cluster balance error using the true class labels, regardless of the nodes' spatial locations. Under both $(iv)$ and $(v)$, our method consistently achieves lower error across all sampling sizes. It is worth noting that $\mathcal{B}(S)$ increases with the size of $S$ by definition in \equationref{eq:balance}: for a fixed graph, the denominators $n_a$ and $n$ are constants, so a larger sampling size $T$ naturally yields a larger error.  The code to reproduce our experimental results is available at \url{https://github.com/zzzzms/GreedyGraphCoresetSelection}.

\begin{table}[t]
\centering
\caption{Average Cluster Balance Error $\mathcal{B}(S)$ on Cora with different sampling sizes $S$. }
\label{tab:mean_estimation_real1}
\vspace{2mm}
\begin{tabular}{lcccccc}
\toprule
     & 
Random Node &
Page Rank &
Degree Based & 
MHRW &  Uniform & GGCS (ours) \\
\midrule
$|S|=20$ & 0.058 & 0.063 & 0.059 & 0.180 & 0.060 & \textbf{0.056} 
\\
$|S|=30$ & 0.049 & 0.049  & 0.049 & 0.173 & 0.051 & \textbf{0.047} \\
$|S|=50$ & 0.038 & 0.038 & 0.040 & 0.157 & 0.038 & \textbf{0.036} \\
$|S|=100$ & 0.027 & 0.026  & 0.028 & 0.122 & 0.027 & \textbf{0.024}  \\
$|S|=300$ & 0.014 & 0.015 & 0.016  & 0.073 & 0.015  & \textbf{0.012}  \\
\bottomrule
\end{tabular}
\end{table}


\section{Conclusion and Future Work}
\label{sec:conclusion}
We presented a scalable graph coreset algorithm based on minimum inner
product selection over random column subsets of the Laplacian, requiring
only linear memory in the number of nodes and never accessing the full
graph. Under the stochastic block model with balanced degree condition, we proved that the algorithm achieves sampling proportional to
cluster size without knowledge of cluster labels, and that this proportional
coverage guarantees the error for mean estimation for any band-limited graph
signal in the Paley-Wiener space defined by the spectral gap of the
Laplacian, with the error decaying as inter-cluster connectivity
weakens. 
Several future research directions are worth pursuing: extending the theoretical analysis to incorporate the squared Laplacian and relaxing the balanced degree assumption; generalizing the framework to weighted or degree-corrected graphs; and adapting the algorithm to dynamic graphs in a streaming setting.

\acks{The authors would like to thank the anonymous reviewers for their valuable feedback and constructive suggestions that improved the quality of the paper.}

\bibliography{references}

@article{ortega2018graph,
  title={Graph signal processing: Overview, challenges, and applications},
  author={Ortega, Antonio and Frossard, Pascal and Kova{\v{c}}evi{\'c}, Jelena and Moura, Jos{\'e} MF and Vandergheynst, Pierre},
  journal={Proceedings of the IEEE},
  volume={106},
  number={5},
  pages={808--828},
  year={2018},
  publisher={IEEE}
}

@article{mateos2019connecting,
  title={Connecting the dots: Identifying network structure via graph signal processing},
  author={Mateos, Gonzalo and Segarra, Santiago and Marques, Antonio G and Ribeiro, Alejandro},
  journal={IEEE Signal Processing Magazine},
  volume={36},
  number={3},
  pages={16--43},
  year={2019},
  publisher={IEEE}
}

@article{newman2002random,
  title={Random graph models of social networks},
  author={Newman, Mark EJ and Watts, Duncan J and Strogatz, Steven H},
  journal={Proceedings of the National Academy of Sciences},
  volume={99},
  number={suppl\_1},
  pages={2566--2572},
  year={2002},
  publisher={National Academy of Sciences}
}

@inproceedings{myers2014information,
  title={Information network or social network? {T}he structure of the {T}witter follow graph},
  author={Myers, Seth A and Sharma, Aneesh and Gupta, Pankaj and Lin, Jimmy},
  booktitle={Proceedings of the 23rd International Conference on World Wide Web},
  pages={493--498},
  year={2014}
}

@article{pavlopoulos2011using,
  title={Using graph theory to analyze biological networks},
  author={Pavlopoulos, Georgios A and Secrier, Maria and Moschopoulos, Charalampos N and Soldatos, Theodoros G and Kossida, Sophia and Aerts, Jan and Schneider, Reinhard and Bagos, Pantelis G},
  journal={BioData Mining},
  volume={4},
  number={1},
  pages={10},
  year={2011},
  publisher={BioMed Central}
}

@article{aittokallio2006graph,
  title={Graph-based methods for analysing networks in cell biology},
  author={Aittokallio, Tero and Schwikowski, Benno},
  journal={Briefings in Bioinformatics},
  volume={7},
  number={3},
  pages={243--255},
  year={2006},
  publisher={Oxford University Press}
}

@book{lippmann2017public,
  title={Public Opinion},
  author={Lippmann, Walter},
  year={2017},
  publisher={Routledge}
}

@inproceedings{yu2005real,
  title={Real-time forest fire detection with wireless sensor networks},
  author={Yu, Liyang and Wang, Neng and Meng, Xiaoqiao},
  booktitle={Proceedings of the 2005 International Conference on Wireless Communications, Networking and Mobile Computing},
  volume={2},
  pages={1214--1217},
  year={2005},
  publisher={IEEE}
}

@inproceedings{choi2017gram,
  title={{GRAM}: Graph-based attention model for healthcare representation learning},
  author={Choi, Edward and Bahadori, Mohammad Taha and Song, Le and Stewart, Walter F and Sun, Jimeng},
  booktitle={Proceedings of the 23rd ACM SIGKDD International Conference on Knowledge Discovery and Data Mining},
  pages={787--795},
  year={2017}
}

@inproceedings{anis2014towards,
  title={Towards a sampling theorem for signals on arbitrary graphs},
  author={Anis, Aamir and Gadde, Akshay and Ortega, Antonio},
  booktitle={Proceedings of the 2014 IEEE International Conference on Acoustics, Speech and Signal Processing (ICASSP)},
  pages={3864--3868},
  year={2014},
  publisher={IEEE}
}

@inproceedings{narang2013signal,
  title={Signal processing techniques for interpolation in graph structured data},
  author={Narang, Sunil K and Gadde, Akshay and Ortega, Antonio},
  booktitle={Proceedings of the 2013 IEEE International Conference on Acoustics, Speech and Signal Processing (ICASSP)},
  pages={5445--5449},
  year={2013},
  publisher={IEEE}
}

@article{marques2015sampling,
  title={Sampling of graph signals with successive local aggregations},
  author={Marques, Antonio G and Segarra, Santiago and Leus, Geert and Ribeiro, Alejandro},
  journal={IEEE Transactions on Signal Processing},
  volume={64},
  number={7},
  pages={1832--1843},
  year={2015},
  publisher={IEEE}
}

@article{ding2024spectral,
  title={Spectral greedy coresets for graph neural networks},
  author={Ding, Mucong and He, Yinhan and Li, Jundong and Huang, Furong},
  journal={arXiv preprint arXiv:2405.17404},
  year={2024}
}

@inproceedings{mirzasoleiman2020coresets,
  title={Coresets for data-efficient training of machine learning models},
  author={Mirzasoleiman, Baharan and Bilmes, Jeff and Leskovec, Jure},
  booktitle={Proceedings of the 37th International Conference on Machine Learning (ICML)},
  pages={6950--6960},
  year={2020},
  publisher={PMLR}
}

@article{sakiyama2019eigendecomposition,
  title={Eigendecomposition-free sampling set selection for graph signals},
  author={Sakiyama, Akie and Tanaka, Yuichi and Tanaka, Toshihisa and Ortega, Antonio},
  journal={IEEE Transactions on Signal Processing},
  volume={67},
  number={10},
  pages={2679--2692},
  year={2019},
  publisher={IEEE}
}

@article{cloninger2021low,
  title={A low discrepancy sequence on graphs},
  author={Cloninger, Alex and Mhaskar, Hrushikesh N},
  journal={Journal of Fourier Analysis and Applications},
  volume={27},
  number={5},
  pages={76},
  year={2021},
  publisher={Springer}
}

@inproceedings{vahidian2020coresets,
  title={Coresets for estimating means and mean square error with limited greedy samples},
  author={Vahidian, Saeed and Mirzasoleiman, Baharan and Cloninger, Alexander},
  booktitle={Proceedings of the 36th Conference on Uncertainty in Artificial Intelligence (UAI)},
  pages={350--359},
  year={2020},
  publisher={PMLR}
}

@article{puy2018random,
  title={Random sampling of bandlimited signals on graphs},
  author={Puy, Gilles and Tremblay, Nicolas and Gribonval, R{\'e}mi and Vandergheynst, Pierre},
  journal={Applied and Computational Harmonic Analysis},
  volume={44},
  number={2},
  pages={446--475},
  year={2018},
  publisher={Elsevier}
}

@article{perraudin2018global,
  title={Global and local uncertainty principles for signals on graphs},
  author={Perraudin, Nathanael and Ricaud, Benjamin and Shuman, David I and Vandergheynst, Pierre},
  journal={APSIPA Transactions on Signal and Information Processing},
  volume={7},
  number={1},
  pages={1--26},
  year={2018},
  publisher={Cambridge University Press}
}

@inproceedings{campbell2018bayesian,
  title={Bayesian coreset construction via greedy iterative geodesic ascent},
  author={Campbell, Trevor and Broderick, Tamara},
  booktitle={Proceedings of the 35th International Conference on Machine Learning (ICML)},
  pages={698--706},
  year={2018},
  publisher={PMLR}
}

@inproceedings{balcilar2021analyzing,
  title={Analyzing the expressive power of graph neural networks in a spectral perspective},
  author={Balcilar, Muhammet and Renton, Guillaume and H{\'e}roux, Pierre and Ga{\"u}z{\`e}re, Benoit and Adam, S{\'e}bastien and Honeine, Paul},
  booktitle={Proceedings of the 9th International Conference on Learning Representations (ICLR)},
  year={2021}
}

@article{davis1994adaptive,
  title={Adaptive time-frequency decompositions},
  author={Davis, Geoffrey M and Mallat, Stephane G and Zhang, Zhifeng},
  journal={Optical Engineering},
  volume={33},
  number={7},
  pages={2183--2191},
  year={1994},
  publisher={SPIE}
}

@inproceedings{pati1993orthogonal,
  title={Orthogonal matching pursuit: Recursive function approximation with applications to wavelet decomposition},
  author={Pati, Yagyensh Chandra and Rezaiifar, Ramin and Krishnaprasad, Perinkulam Sambamurthy},
  booktitle={Proceedings of the 27th Asilomar Conference on Signals, Systems and Computers},
  pages={40--44},
  year={1993},
  publisher={IEEE}
}

@inproceedings{shen2023graph,
  title={Graph-based Semi-supervised Local Clustering with Few Labeled Nodes},
  author={Shen, Zhaiming and Lai, Ming-Jun and Li, Sheng},
  booktitle={Proceedings of the 32nd International Joint Conference on Artificial Intelligence (IJCAI)},
  year={2023}
}

@inproceedings{shen2025advancing,
  title={Advancing Local Clustering on Graphs via Compressive Sensing: Semi-supervised and Unsupervised Methods},
  author={Shen, Zhaiming and Kang, Sung Ha},
  booktitle={Proceedings of the 1st Conference on Topology, Algebra, and Geometry in Data Science (TAG-DS 2025)},
  pages={126--146},
  year={2025},
  volume={321},
  publisher={PMLR}
}

@article{lai2023compressed,
  title={A compressed sensing based least squares approach to semi-supervised local cluster extraction},
  author={Lai, Ming-Jun and Shen, Zhaiming},
  journal={Journal of Scientific Computing},
  volume={94},
  number={3},
  pages={63},
  year={2023},
  publisher={Springer}
}

@article{sen2008collective,
  title={Collective classification in network data},
  author={Sen, Prithviraj and Namata, Galileo and Bilgic, Mustafa and Getoor, Lise and Galligher, Brian and Eliassi-Rad, Tina},
  journal={AI Magazine},
  volume={29},
  number={3},
  pages={93--106},
  year={2008},
  publisher={Association for the Advancement of Artificial Intelligence}
}

@article{stumpf2005subnets,
  title={Subnets of scale-free networks are not scale-free: sampling properties of networks},
  author={Stumpf, Michael PH and Wiuf, Carsten and May, Robert M},
  journal={Proceedings of the National Academy of Sciences},
  volume={102},
  number={12},
  pages={4221--4224},
  year={2005},
  publisher={National Academy of Sciences}
}

@inproceedings{leskovec2006sampling,
  title={Sampling from large graphs},
  author={Leskovec, Jure and Faloutsos, Christos},
  booktitle={Proceedings of the 12th ACM SIGKDD International Conference on Knowledge Discovery and Data Mining},
  pages={631--636},
  year={2006},
  publisher={ACM}
}

@article{adamic2001search,
  title={Search in power-law networks},
  author={Adamic, Lada A and Lukose, Rajan M and Puniyani, Amit R and Huberman, Bernardo A},
  journal={Physical Review E},
  volume={64},
  number={4},
  pages={046135},
  year={2001},
  publisher={American Physical Society}
}

@inproceedings{hubler2008metropolis,
  title={Metropolis algorithms for representative subgraph sampling},
  author={H{\"u}bler, Christian and Kriegel, Hans-Peter and Borgwardt, Karsten and Ghahramani, Zoubin},
  booktitle={Proceedings of the 2008 IEEE International Conference on Data Mining (ICDM)},
  pages={283--292},
  year={2008},
  publisher={IEEE}
}

@inproceedings{stutzbach2006unbiased,
  title={On unbiased sampling for unstructured peer-to-peer networks},
  author={Stutzbach, Daniel and Rejaie, Reza and Duffield, Nick and Sen, Subhabrata and Willinger, Walter},
  booktitle={Proceedings of the 6th ACM SIGCOMM Conference on Internet Measurement},
  pages={27--40},
  year={2006},
  publisher={ACM}
}

@inproceedings{rozemberczki2020little,
  title={{Little Ball of Fur}: A {Python} library for graph sampling},
  author={Rozemberczki, Benedek and Kiss, Oliver and Sarkar, Rik},
  booktitle={Proceedings of the 29th ACM International Conference on Information and Knowledge Management (CIKM)},
  pages={3133--3140},
  year={2020},
  publisher={ACM}
}

@article{shchur2018pitfalls,
  author  = {Shchur, Oleksandr and Mumme, Maximilian and Bojchevski, Aleksandar and G{\"u}nnemann, Stephan},
  title   = {Pitfalls of Graph Neural Network Evaluation},
  journal = {arXiv preprint arXiv:1811.05868},
  year    = {2018}
}

@article{chen2018network,
  author  = {Chen, Kehui and Lei, Jing},
  title   = {Network Cross-Validation for Determining the Number of Communities in Network Data},
  journal = {Journal of the American Statistical Association},
  volume  = {113},
  number  = {521},
  pages   = {241--251},
  year    = {2018},
  publisher = {Taylor \& Francis}
}

@article{lai2020compressive,
  title={Compressive sensing for cut improvement and local clustering},
  author={Lai, Ming-Jun and Mckenzie, Daniel},
  journal={SIAM Journal on Mathematics of Data Science},
  volume={2},
  number={2},
  pages={368--395},
  year={2020},
  publisher={SIAM}
}

\appendix

\section{Model Assumptions} \label{sec:assump}
 
We model the graph $G$ as a draw from the Stochastic Block Model with parameters
$n$, $K$, $\{n_a\}_{a=1}^K$, $\{p_a\}_{a=1}^K$,
$\{q_{ab}\}_{a,b=1,a\neq b}^{K}$. Our theoretical analysis requires the following assumptions.
 
\begin{assumption}[Stochastic Block Model]
\label{ass:sbm}
The node set $V$ is partitioned into $K$ disjoint clusters
$V_1, \ldots, V_K$ of sizes $n_1, \ldots, n_K$ with
$\sum_{a=1}^K n_a = n$. Edges are drawn independently with $\mathbb{P}[(i,j) \in E]=p_a$ if $i, j \in V_a$, and $\mathbb{P}[(i,j) \in E]=q_{ab}$ if $i \in V_a,\, j \in V_b,\, a \neq b$.
\end{assumption}

 
\begin{assumption}[Sparse Inter-Cluster Connectivity]
\label{ass:sparse}
The inter-cluster edge probabilities satisfy
$q_{\max} = o(p_{\min})$, where $q_{\max}:=\max_{a,b=1,\cdots,K}\{q_{ab}\}$ and $p_{\min}:=\min_{a=1,\cdots,K}\{p_{a}\}$.

\end{assumption}
 
\begin{assumption}[Balanced Intra-Cluster Degree]
\label{ass:balanced}
There exists a sequence $\omega_n \to \infty$  such that the intra-cluster edge
probabilities satisfy:
\begin{equation}
    n_a p_a = \omega_n \qquad \forall\, a \in \{1, \ldots, K\},
    \label{eq:balanced}
\end{equation}
so $p_a = \omega_n/n_a$.
This keeps the expected intra-cluster degree equal across all clusters
while allowing $\omega_n$ to grow with $n$.
Typical choices include $\omega_n = \Theta(\log n)$ (sparse regime)
or $\omega_n = \Theta(\sqrt{n})$ (moderately dense regime).
\end{assumption}

\section{Deferred Proofs and Other Useful Lemmas}\label{sec:proofs}

In the subsequent subsections, we first establish an exact expression
for the Laplacian column inner products
(\lemmaref{lem:l2_exact}), then characterize the result for balanced cluster proportion (\theoremref{thm:residual}), and finally prove
the error bound for the mean value estimation of band-limited functions defined on a graph using the selected coreset (\theoremref{thm:error_bound}).

\subsection{Proof of \lemmaref{lem:l2_exact}}

\begin{proof}[Proof of \lemmaref{lem:l2_exact}]
\equationref{eq:l2_exact2}:
Split $\sum_k L_{ik}L_{kj}$ into $k=i$, $k=j$, and $k\neq i,j$,
giving $-d_i\mathbf{1}[(i,j)\in E]$,
$-d_j\mathbf{1}[(i,j)\in E]$, and
$\sum_{k\neq i,j}\mathbf{1}[(i,k)\in E]\mathbf{1}[(j,k)\in E]
= |\mathcal{N}(i)\cap\mathcal{N}(j)|$ respectively.
 
\equationref{eq:l2_simplified} Case 1 ($i,j\in V_a$):
Each of the $n_a-2$ within-cluster nodes contributes $p_a^2$ and each
of the $n_c$ cross-cluster nodes contributes $q_{ac}^2$ to
$\mathbb{E}[|\mathcal{N}(i)\cap\mathcal{N}(j)|]$, giving
$(n_a-2)p_a^2 + \sum_{c\neq a}n_cq_{ac}^2$.
The edge correction is
$p_a\cdot 2((n_a-1)p_a + \sum_{c\neq a}n_cq_{ac})$.
Subtracting and collecting the leading terms:
$(n_a-2)p_a^2 - 2(n_a-1)p_a^2 = -n_ap_a^2 + O(p_a^2)$,
with the cross-cluster residual $O(np_aq_{\max})$,
giving the first case of~\equationref{eq:l2_simplified}.
 
\equationref{eq:l2_simplified} Case 2 ($i\in V_a$, $j\in V_b$, $a\neq b$):
The common neighbor sum splits as
$(n_a-1)p_aq_{ab} + (n_b-1)q_{ab}p_b
+ \sum_{c\neq a,b}n_cq_{ac}q_{bc}$.
The degree correction is
$q_{ab}((n_a-1)p_a + \sum_{c\neq a}n_cq_{ac}
+(n_b-1)p_b+\sum_{c\neq b}n_cq_{bc})$.
The $O(\omega_n q_{ab})$ terms from $c=a$ and $c=b$ cancel
exactly, leaving $\sum_{c\neq a,b}n_cq_{ac}q_{bc}
+ O(nq_{\max}^2) = O(nq_{\max}^2)$,
and $O(nq_{\max}^2) = o(n_{\min}p_{\min}^2)$
since $q_{\max} = o(p_{\min})$.
\end{proof}

\subsection{Proof of \theoremref{thm:residual}}

\begin{proof}[Proof of \theoremref{thm:residual}]
Since $L_{:,j}^\top\mathbf{b}_t = -\sum_{s=1}^{t}(L^2)_{ji_s}$
and $\mathcal{X}_t$ is deterministic, taking conditional expectation
over $G$ and splitting by cluster:
\begin{equation*}
    \mathbb{E}\!\left[L_{:,j}^\top\mathbf{b}_t\mid\mathcal{X}_t\right]
    = -\sum_{i_s\in V_a}\mathbb{E}[(L^2)_{ji_s}]
    - \sum_{c\neq a}\sum_{i_s\in V_c}\mathbb{E}[(L^2)_{ji_s}].
\end{equation*}
Substituting $\mathbb{E}[(L^2)_{ji_s}] = -n_ap_a^2 + O(np_aq_{\max})$
for $i_s\in V_a$ and $\mathbb{E}[(L^2)_{ji_s}] = O(nq_{\max}^2)$
for $i_s\in V_c$, $c\neq a$, from~\eqref{eq:l2_simplified}
gives~\eqref{eq:residual_exact}--\eqref{eq:eta_bound}.

Furthermore, suppose $\frac{m_a^{(t)}}{n_a}$ is the smallest among all $a=1,\cdots,K$ up to tolerance $|\eta_a^{(t)}|$, i.e.,
\begin{equation*}
    \frac{m_a^{(t)}}{n_a}+\frac{|\eta_a^{(t)}|}{\omega_n^2}< \frac{m_b^{(t)}}{n_b}-\frac{|\eta_b^{(t)}|}{\omega_n^2}.
\end{equation*}
Multiplying the above equation both sides by $\omega_n^2 = n_a^2p_a^2=n_b^2p_b^2$, gives
\begin{equation*}
    \mathbb{E}\!\left[L_{:,j_a}^\top \mathbf{b}_t
    \,\middle|\, \mathcal{X}_t\right]<\mathbb{E}\!\left[L_{:,j_b}^\top \mathbf{b}_t
    \,\middle|\, \mathcal{X}_t\right]
\end{equation*}
for any $j_a\in V_a$ and $j_b\in V_b$. Therefore, \algorithmref{alg_omp} selects a node from $V_a$ in the $(t+1)$-th iteration.
\end{proof}

\subsection{Proof of \theoremref{thm:error_bound}}

\begin{proof}[Proof of \theoremref{thm:error_bound}]
Let us define $\delta_k(S):=\frac{1}{T}\sum_{i\in S}u_k(i)-\frac{1}{n}\sum_{i\in V}u_k(i)$.
Since $f \in PW_\omega(L)$, $\hat{f}_k = 0$ for $k \geq K$, so the
estimation error reduces to:
\begin{equation}
    \hat{\mu}_S - \mu
    = \sum_{k=0}^{K-1}\hat{f}_k\,\delta_k(S)=\sum_{k=1}^{K-1}\hat{f}_k\,\delta_k(S).
    \label{eq:error_expand}
\end{equation}
Notice that $\delta_0(S)=0$, as $u_0$ is a constant vector across all nodes. Also, notice that for $k\geq 1$, $\delta_k(S) = \frac{1}{T}\sum_{i\in S}u_k(i)$ as
$\frac{1}{n}\sum_{i\in V}u_k(i) = 0$.
 
\paragraph{Step 1: Decompose $u_k$ via Wedin \textbf{sin}${\Theta}$ Lemma.}
Write each eigenvector as:
\begin{equation}
    u_k = \tilde{u}_k + r_k,
    \label{eq:dk_decomp}
\end{equation}
where $\tilde{u}_k = \sum_{j=1}^K \Phi_{jk} u_j^{\mathrm{block}}$ is
the rotated block-diagonal approximation (exactly cluster-constant
within each cluster) and $r_k = u_k - \tilde{u}_k$ is the perturbation
residual. By Lemma~\ref{lem:davis_kahan}:
\begin{equation}
    \sum_{k=0}^{K-1}\|r_k\|_2^2
    = \|U_K - U_K^{\mathrm{block}}\Phi\|_F^2
    \leq \frac{8K\,\|E\|_2^2}{\Delta^2}.
    \label{eq:residual_bound}
\end{equation}
 
\paragraph{Step 2: Bound the discrepancy $\delta_k(S)$.}
Substituting~\eqref{eq:dk_decomp} into $\delta_k(S)$:
\begin{align}
    \delta_k(S)
    &= \frac{1}{T}\sum_{i\in S}\tilde{u}_k(i)
     + \frac{1}{T}\sum_{i\in S}r_k(i)
    \triangleq \tilde{\delta}_k(S) + \rho_k(S).
    \label{eq:delta_decomp}
\end{align}
 
For $\tilde{\delta}_k(S)$: since $\tilde{u}_k$ is exactly
cluster-constant with value $\tilde{c}_k^{(a)}$ for $i\in V_a$:
\begin{equation*}
    |\tilde{\delta}_k(S)|
    = \left|\sum_{a=1}^K \tilde{c}_k^{(a)}\epsilon_a\right|
    \leq \epsilon\sum_{a=1}^K|\tilde{c}_k^{(a)}|
    \leq \epsilon\sqrt{K}\,\|\tilde{u}_k\|_\infty
    \leq \epsilon\sqrt{K},
\end{equation*}
where the last step uses $\|\tilde{u}_k\|_\infty \leq
\|\tilde{u}_k\|_2 = 1$.
 
For $\rho_k(S)$: by the Cauchy-Schwarz inequality and
$|S\cap V_a| \leq T$:
\begin{equation*}
    |\rho_k(S)|
    = \left|\frac{1}{T}\sum_{i\in S}r_k(i)\right|
    \leq \frac{1}{T}\sum_{i\in S}|r_k(i)|
    \leq \frac{\sqrt{T}\,\|r_k\|_2}{T}
    = \frac{\|r_k\|_2}{\sqrt{T}},
\end{equation*}
Therefore, we get
\begin{equation}
    |\delta_k(S)|
    \leq \epsilon\sqrt{K}
       + \frac{\|r_k\|_2}{\sqrt{T}}.
    \label{eq:delta_bound}
\end{equation}
 
\paragraph{Step 3: Combine via Cauchy-Schwarz.}
Applying Cauchy-Schwarz to~\eqref{eq:error_expand}:
\begin{align*}
    |\hat{\mu}_S - \mu|
    &\leq \left(\sum_{k=1}^{K-1}\hat{f}_k^2\right)^{1/2}
          \left(\sum_{k=1}^{K-1}\delta_k(S)^2\right)^{1/2} \\
    &\leq \sqrt{2K}\,\|f\|_2\,
          \sqrt{\left(\epsilon^2(K-1) + \frac{8\|E\|^2_2}
          {\Delta^2 \cdot T}\right)} \\
          &\leq \sqrt{K}\,\|f\|_2\,\left(\epsilon\sqrt{2K}+\frac{4\|E\|_2}{\Delta\cdot \sqrt{T}}\right),
\end{align*}
which gives~\eqref{eq:error_bound}.
\end{proof}

\subsection{Wedin \textbf{sin}${\Theta}$ Lemma}
\begin{lemma}[Wedin \textbf{sin}${\Theta}$ Lemma]
\label{lem:davis_kahan}
Let $L = L_{\mathrm{block}} + E$ where $L_{\mathrm{block}}$ is the
block diagonal Laplacian under perfect cluster separation ($q=0$) and
$E$ contains the cross-cluster edge contributions. Let $U_K$ and
$U_K^{\mathrm{block}}$ be the matrices of the first $K$ eigenvectors of
$L$ and $L_{\mathrm{block}}$ respectively, and let
$\Delta = \lambda_K - \lambda_{K-1}$ be the spectral gap of $L$.
Then there exists an orthogonal matrix $\Phi \in \mathbb{R}^{K\times K}$
such that:
\begin{equation}
    \|U_K - U_K^{\mathrm{block}}\Phi\|_F
    \leq \frac{2\sqrt{2K}\,\|E\|_2}{\Delta}.
    \label{eq:dk_bound}
\end{equation}
\end{lemma}

\begin{proof}
    We refer to Lemma 7 in \citet{chen2018network} for a formal proof.
\end{proof}

\section{Other Details and Deferred Results on Experiments}\label{sec:exp_details}

\subsection{Benchmark Methods.}  We compare our method against baseline sampling schemes, including Random Node~\citep{stumpf2005subnets}, Page Rank~\citep{leskovec2006sampling}, Degree-based~\citep{adamic2001search}, Metropolis-Hastings Random Walk~\citep{hubler2008metropolis, stutzbach2006unbiased}, and uniform node sampling. For convenience, implementations of these baselines are available via the Python library of~\citet{rozemberczki2020little}. 

All simulations are repeated over $100$ independent runs. For synthetic data, each run generates a random graph and a random sample from that generated graph. For real data, each run generates a random sample from a fixed graph.  





\subsection{Additional Experimental Results}

\begin{table}[t]
\centering
\caption{Average results on sampling $100$ nodes from the SBM graph with three clusters of
sizes $n_1 = 1000$, $n_2=1000$, $n_3 = 1000$, and intra-cluster edge
probabilities $p_1, p_2, p_3 = 0.01, 0.01, 0.01$, and $q_{ij}=0.001$ for $i,j=1,2,3$.
The signal $f$ is defined in \eqref{eq:pw_random}.  }
\label{tab:mean_estimation_sbm2}
\begin{tabular}{ccccccc}
\toprule
     & 
Random Node &
Page Rank &
Degree Based & 
MHRW &  Uniform & GGCS (ours) \\
\midrule
$\epsilon(S)$ & 0.053   & 0.054 & 0.048 & 0.102 & 0.051 & \textbf{0.046} \\
$\mathcal{B}(S)$  & 0.037 & 0.038 & 0.039 & 0.096 & 0.040 & \textbf{0.035}
\\
\bottomrule
\end{tabular}
\end{table}

\begin{table}[t]
\centering
\caption{Average Cluster Balance Error $\mathcal{B}(S)$ on Amazon Photo dataset with different sampling sizes $S$. }
\label{tab:mean_estimation_real2}
\begin{tabular}{lcccccc}
\toprule
     & 
Random Node &
Page Rank &
Degree Based & 
MHRW &  Uniform & GGCS (ours) \\
\midrule
$|S|=20$ & 0.054 & 0.055 & 0.066 & 0.180 & 0.056 & \textbf{0.050}
\\
$|S|=30$ & 0.047 & 0.047 & 0.057 & 0.171 & 0.044 & \textbf{0.041} \\
$|S|=50$ & 0.036 & 0.036 & 0.048 & 0.157 & 0.035 &  \textbf{0.032} \\
$|S|=100$ & 0.025 & 0.026 & 0.043 & 0.142 & 0.024 & \textbf{0.022} \\
$|S|=300$ & 0.014 & 0.015  & 0.034 & 0.113 & 0.014 & \textbf{0.012} \\
\bottomrule
\end{tabular}
\end{table}

\end{document}